\documentclass[11pt]{article}

\usepackage{amssymb}
\usepackage{amsmath}
\usepackage{amsthm}
\usepackage{mathdots}
\usepackage{xcolor}
\usepackage{graphicx}
\usepackage{tikz-cd}
\usepackage{pgf,acadpgf}
\usepackage{enumerate}

\long\def\commentout#1{}

\numberwithin{equation}{section}

\newtheoremstyle{slplain}
  {\topsep}
  {\topsep}
  {\slshape}
  {0pt}
  {\bfseries}
  {.}
  {0.5em}
  {}

\theoremstyle{slplain}

  \newtheorem*{THMC}{Theorem C}
  \newtheorem*{THMD}{Theorem D}
  \newtheorem*{THME}{Theorem E}
  \newtheorem*{LEMA}{Lemma A}
  \newtheorem*{LEMB}{Lemma B}

\theoremstyle{definition}

\newcommand{\Fraisse}{Fra\"\i ss\'e}

\newcommand{\lex}{\le_{\mathit{lex}}}
\newcommand{\alex}{\le_{\mathit{alex}}}
\newcommand{\refl}{{\mathit{refl}}}

\newcommand\nlongrightarrow{\longrightarrow\kern -1.45em/\kern 0.9em}

\renewcommand{\preceq}{\preccurlyeq}

\renewcommand{\le}{\leqslant}
\renewcommand{\ge}{\geqslant}

\newcommand{\0}{\varnothing}

\renewcommand{\phi}{\varphi}
\renewcommand{\epsilon}{\varepsilon}

\newcommand{\KK}{\mathbf{K}}

\newcommand{\NN}{\mathbb{N}}

\newcommand{\QQ}{\mathbb{Q}}

\newcommand{\Boxed}[1]{\mbox{$#1$}}

\newcommand{\Age}{\mathrm{Age}}

\newcommand{\calA}{\mathcal{A}}
\newcommand{\calB}{\mathcal{B}}
\newcommand{\calC}{\mathcal{C}}
\newcommand{\calD}{\mathcal{D}}

\newcommand{\calK}{\mathcal{K}}

\newcommand{\calQ}{\mathcal{Q}}

\newcommand{\calS}{\mathcal{S}}

\title{Corrigendum and Addendum to\\``\Fraisse's Conjecture and big Ramsey degrees of structures admitting finite monomorphic decomposition''}
\author{%
  Dragan Ma\v sulovi\'c (corresponding author)\\
  University of Novi Sad, Faculty of Sciences\\
  Department of Mathematics and Informatics\\
  Trg Dositeja Obradovi\'ca 3, 21000 Novi Sad, Serbia\\
  e-mail: dragan.masulovic@dmi.uns.ac.rs
\and
  Veljko Tolji\'c\\
  University of Novi Sad, Faculty of Sciences\\
  Department of Mathematics and Informatics\\
  Trg Dositeja Obradovi\'ca 3, 21000 Novi Sad, Serbia\\
  e-mail: veljko.toljic@dmi.uns.ac.rs
}

\begin{document}
\maketitle

\begin{abstract}
  In Section 6 of the paper ``\Fraisse's Conjecture and big Ramsey degrees of structures admitting finite monomorphic decomposition'',
  we applied the methods developed in earlier sections to show that a certain reduct of the generic permutation has
  finite big Ramsey degrees. Unfortunately, this reduct was incorrectly identified as the generic partial order.
  We are grateful to Jan Hubi\v cka for bringing this error to our attention.

  In this note we correct the statements that rely on this misidentification and demonstrate that the reduct in question is
  in fact the generic \emph{2-dimensional} partial order. We emphasize that the arguments presented in Section 6 remain valid,
  with the sole exception of the Claim in the proof of Theorem 6.4, whose role was to (incorrectly) identify the
  reduct of the generic permutation as the generic partial order.

  This correction has an unexpected positive consequence. Rather than reproving a well-known result whose existing proof
  is already notably elegant, this note demonstrates that our general framework can be used to establish that a previously unexplored
  class of generic relational structures has finite big Ramsey degrees. This observation opens a potentially new
  direction for further research in the thriving area of big Ramsey combinatorics.
  
  In the addendum, we combine a recent result by Oudrar and Pouzet with our
  analysis of finite big Ramsey degrees for structures admitting finite monomorphic decomposition
  to characterize the existence of finite Big Ramsey degrees for all countable relational structures
  whose language has a linear order and age has polynomial growth. 
  
  A revised version of the paper 
  ``\Fraisse's Conjecture and big Ramsey degrees of structures admitting finite monomorphic decomposition''
  incorporating all additions and corrections addressed here is available as arXiv:2407.20307.
\end{abstract}

\section*{Corrigendum}

\noindent$\bullet$
Last sentence of the Abstract:
replace ``provide an alternative proof of Hubi\v cka's result'' with ``prove'';
after ``generic'' add ``2-dimensional''.

\bigskip

\noindent$\bullet$
Last paragraph of the Introduction:
replace ``provide an alternative proof of Hubi\v cka's result'' with ``prove'';
after ``generic'' add ``2-di\-men\-si\-o\-nal''; delete reference ``[10]''.

\bigskip

\noindent$\bullet$
Title of Section 6: after ``generic'' add ``2-di\-men\-si\-o\-nal''

\bigskip

\noindent$\bullet$
First paragraph of Section 6:
replace ``provide an alternative proof of the result of Hubi\v cka from~[10]'' with ``prove'';
after ``generic'' add ``2-di\-men\-si\-o\-nal'';
delete ``Instead of Voigt's infinite $\star$-version of the
Graham-Rothschild's theorem (see [19, Theorem~A]),'';
replace ``our'' with ``Our''.

\bigskip

\noindent$\bullet$
Third paragraph of Section 6:
after ``from the rationals $\QQ$ to the generic'' add ``2-dimensional''.

\bigskip

\noindent$\bullet$
Fourth paragraph of Section 6:
after ``We then use the fact that the generic'' add ``2-dimensional'';
after ``from the generic permutation to the generic'' add ``2-dimensional''.

\bigskip

\noindent$\bullet$
Formulation of Theorem 6.4: delete reference ``[10]'';
after ``generic'' add ``2-dimensional''.

\bigskip

\noindent$\bullet$
Proof of Theorem 6.4, second paragraph: delete entire Claim; delete ``Proof. Straightforward.''

\bigskip

\noindent$\bullet$
Proof of Theorem 6.4, last paragraph:
delete ``Therefore,'';
replace ``if'' with ``If'';
after ``is precisely the class of all finite or countably infinite'' add ``2-dimensional'';
replace ``Moreover,'' with ``In what follows, we are going to prove that'';
after ``$(\QQ, \Boxed{\preccurlyeq})$ is the generic'' add ``2-dimensional'';
after ``Theorem 6.3 immediately yields that the generic'' add ``2-dimensional''.

\bigskip

\noindent$\bullet$
After Theorem 6.4 add the following:

\bigskip

The \emph{dimension} of a finite partially ordered set $(P, \Boxed{\preceq})$ was introduced in 1941
by Dushnik and Miller~\cite{dushnik-miller-1941} as the least number of linear extensions of $\preceq$
whose intersection is $\preceq$. Equivalently, the dimension of $(P, \Boxed{\preceq})$ is the least
number $d$ of finite linear orders $(L_1, \le_1)$, \ldots, $(L_d, \le_d)$ such that
$(P, \Boxed{\preceq})$ embeds into $(L_1, \le_1) \times \ldots \times (L_d, \le_d)$ (see \cite{ore-1962}).
A finite partially ordered set $(P, \Boxed{\preceq})$ is \emph{at most 2-dimensional} if its dimension is $\le 2$.

\begin{figure}
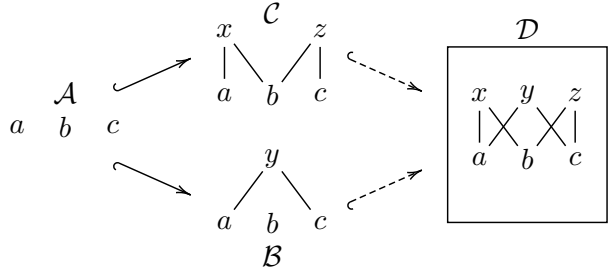

  \centering
\begin{pgfpicture}
  \pgfsetxvec{\pgfpoint{\acadpgfunit}{0pt}}
  \pgfsetyvec{\pgfpoint{0pt}{\acadpgfunit}}
  \pgfsetlinewidth{\acadpgflinewidth}
  \pgftransformshift{\pgfpointxy{-237.5}{-212.5}}

  \begin{pgfscope}
    \pgfpathmoveto{\pgfpointxy{1050.0}{475.0}}
    \pgfpathlineto{\pgfpointxy{1050.0}{525.0}}
    \pgfusepath{stroke}
  \end{pgfscope}
  \begin{pgfscope}
    \pgfpathmoveto{\pgfpointxy{1065.0}{530.0}}
    \pgfpathlineto{\pgfpointxy{1110.0}{470.0}}
    \pgfusepath{stroke}
  \end{pgfscope}
  \begin{pgfscope}
    \pgfpathmoveto{\pgfpointxy{1140.0}{470.0}}
    \pgfpathlineto{\pgfpointxy{1185.0}{530.0}}
    \pgfusepath{stroke}
  \end{pgfscope}
  \begin{pgfscope}
    \pgfpathmoveto{\pgfpointxy{1200.0}{525.0}}
    \pgfpathlineto{\pgfpointxy{1200.0}{475.0}}
    \pgfusepath{stroke}
  \end{pgfscope}
  \begin{pgfscope}
    \pgfpathmoveto{\pgfpointxy{1185.0}{470.0}}
    \pgfpathlineto{\pgfpointxy{1140.0}{530.0}}
    \pgfusepath{stroke}
  \end{pgfscope}
  \begin{pgfscope}
    \pgfpathmoveto{\pgfpointxy{1110.0}{530.0}}
    \pgfpathlineto{\pgfpointxy{1065.0}{470.0}}
    \pgfusepath{stroke}
  \end{pgfscope}
  \begin{pgfscope}
    \pgfpathmoveto{\pgfpointxy{785.0}{370.0}}
    \pgfpathlineto{\pgfpointxy{740.0}{430.0}}
    \pgfusepath{stroke}
  \end{pgfscope}
  \begin{pgfscope}
    \pgfpathmoveto{\pgfpointxy{710.0}{430.0}}
    \pgfpathlineto{\pgfpointxy{665.0}{370.0}}
    \pgfusepath{stroke}
  \end{pgfscope}
  \begin{pgfscope}
    \pgfpathmoveto{\pgfpointxy{650.0}{575.0}}
    \pgfpathlineto{\pgfpointxy{650.0}{625.0}}
    \pgfusepath{stroke}
  \end{pgfscope}
  \begin{pgfscope}
    \pgfpathmoveto{\pgfpointxy{665.0}{630.0}}
    \pgfpathlineto{\pgfpointxy{710.0}{570.0}}
    \pgfusepath{stroke}
  \end{pgfscope}
  \begin{pgfscope}
    \pgfpathmoveto{\pgfpointxy{740.0}{570.0}}
    \pgfpathlineto{\pgfpointxy{785.0}{630.0}}
    \pgfusepath{stroke}
  \end{pgfscope}
  \begin{pgfscope}
    \pgfpathmoveto{\pgfpointxy{800.0}{625.0}}
    \pgfpathlineto{\pgfpointxy{800.0}{575.0}}
    \pgfusepath{stroke}
  \end{pgfscope}
  \begin{pgfscope}
    \pgfpathmoveto{\pgfpointxy{1000.0}{625.0}}
    \pgfpathlineto{\pgfpointxy{1250.0}{625.0}}
    \pgfpathlineto{\pgfpointxy{1250.0}{350.0}}
    \pgfpathlineto{\pgfpointxy{1000.0}{350.0}}
    \pgfpathclose
    \pgfusepath{stroke}
  \end{pgfscope}
  \begin{pgfscope}
    \pgfpathmoveto{\pgfpointxy{484.487}{556.936}}
    \pgfpathlineto{\pgfpointxy{596.987}{606.936}}
    \pgfusepath{stroke}
  \end{pgfscope}
  \begin{pgfscope}
    \pgfpathmoveto{\pgfpointxy{479.802}{567.478}}
    \pgfpatharcaxes{113.962}{293.962}{\pgfpointxy{5.7681}{0.0}}{\pgfpointxy{0.0}{5.7681}}
    \pgfusepath{stroke}
  \end{pgfscope}
  \begin{pgfscope}
    \pgfpathmoveto{\pgfpointxy{573.978}{603.307}}
    \pgfpatharcaxes{263.962}{293.962}{\pgfpointxy{45.0}{0.0}}{\pgfpointxy{0.0}{45.0}}
    \pgfusepath{stroke}
  \end{pgfscope}
  \begin{pgfscope}
    \pgfpathmoveto{\pgfpointxy{596.987}{606.936}}
    \pgfpatharcaxes{113.962}{143.962}{\pgfpointxy{45.0}{0.0}}{\pgfpointxy{0.0}{45.0}}
    \pgfusepath{stroke}
  \end{pgfscope}
  \begin{pgfscope}
    \pgfpathmoveto{\pgfpointxy{484.487}{443.064}}
    \pgfpathlineto{\pgfpointxy{596.987}{393.064}}
    \pgfusepath{stroke}
  \end{pgfscope}
  \begin{pgfscope}
    \pgfpathmoveto{\pgfpointxy{484.487}{443.064}}
    \pgfpatharcaxes{66.0375}{246.038}{\pgfpointxy{5.7681}{0.0}}{\pgfpointxy{0.0}{5.7681}}
    \pgfusepath{stroke}
  \end{pgfscope}
  \begin{pgfscope}
    \pgfpathmoveto{\pgfpointxy{596.987}{393.064}}
    \pgfpatharcaxes{66.0375}{96.0375}{\pgfpointxy{45.0}{0.0}}{\pgfpointxy{0.0}{45.0}}
    \pgfusepath{stroke}
  \end{pgfscope}
  \begin{pgfscope}
    \pgfpathmoveto{\pgfpointxy{578.875}{407.711}}
    \pgfpatharcaxes{216.038}{246.038}{\pgfpointxy{45.0}{0.0}}{\pgfpointxy{0.0}{45.0}}
    \pgfusepath{stroke}
  \end{pgfscope}
  \begin{pgfscope}
    \pgfsetdash{{2.5pt}{1.5pt}}{0pt}
    \pgfpathmoveto{\pgfpointxy{846.987}{605.564}}
    \pgfpathlineto{\pgfpointxy{959.487}{555.564}}
    \pgfusepath{stroke}
  \end{pgfscope}
  \begin{pgfscope}
    \pgfpathmoveto{\pgfpointxy{851.672}{616.106}}
    \pgfpatharcaxes{66.0375}{246.038}{\pgfpointxy{5.7681}{0.0}}{\pgfpointxy{0.0}{5.7681}}
    \pgfusepath{stroke}
  \end{pgfscope}
  \begin{pgfscope}
    \pgfpathmoveto{\pgfpointxy{959.487}{555.564}}
    \pgfpatharcaxes{66.0375}{96.0375}{\pgfpointxy{45.0}{0.0}}{\pgfpointxy{0.0}{45.0}}
    \pgfusepath{stroke}
  \end{pgfscope}
  \begin{pgfscope}
    \pgfpathmoveto{\pgfpointxy{941.375}{570.211}}
    \pgfpatharcaxes{216.038}{246.038}{\pgfpointxy{45.0}{0.0}}{\pgfpointxy{0.0}{45.0}}
    \pgfusepath{stroke}
  \end{pgfscope}
  \begin{pgfscope}
    \pgfsetdash{{2.5pt}{1.5pt}}{0pt}
    \pgfpathmoveto{\pgfpointxy{846.987}{381.936}}
    \pgfpathlineto{\pgfpointxy{959.487}{431.936}}
    \pgfusepath{stroke}
  \end{pgfscope}
  \begin{pgfscope}
    \pgfpathmoveto{\pgfpointxy{846.987}{381.936}}
    \pgfpatharcaxes{113.962}{293.962}{\pgfpointxy{5.7681}{0.0}}{\pgfpointxy{0.0}{5.7681}}
    \pgfusepath{stroke}
  \end{pgfscope}
  \begin{pgfscope}
    \pgfpathmoveto{\pgfpointxy{936.478}{428.307}}
    \pgfpatharcaxes{263.962}{293.962}{\pgfpointxy{45.0}{0.0}}{\pgfpointxy{0.0}{45.0}}
    \pgfusepath{stroke}
  \end{pgfscope}
  \begin{pgfscope}
    \pgfpathmoveto{\pgfpointxy{959.487}{431.936}}
    \pgfpatharcaxes{113.962}{143.962}{\pgfpointxy{45.0}{0.0}}{\pgfpointxy{0.0}{45.0}}
    \pgfusepath{stroke}
  \end{pgfscope}
  \pgftext[at={\pgfpointxy{1050.0}{450.0}}]{$a$}
  \pgftext[at={\pgfpointxy{1125.0}{450.0}}]{$b$}
  \pgftext[at={\pgfpointxy{1200.0}{450.0}}]{$c$}
  \pgftext[at={\pgfpointxy{1050.0}{550.0}}]{$x$}
  \pgftext[at={\pgfpointxy{1125.0}{550.0}}]{$y$}
  \pgftext[at={\pgfpointxy{1200.0}{550.0}}]{$z$}
  \pgftext[at={\pgfpointxy{650.0}{350.0}}]{$a$}
  \pgftext[at={\pgfpointxy{725.0}{350.0}}]{$b$}
  \pgftext[at={\pgfpointxy{800.0}{350.0}}]{$c$}
  \pgftext[at={\pgfpointxy{725.0}{450.0}}]{$y$}
  \pgftext[at={\pgfpointxy{650.0}{550.0}}]{$a$}
  \pgftext[at={\pgfpointxy{725.0}{550.0}}]{$b$}
  \pgftext[at={\pgfpointxy{800.0}{550.0}}]{$c$}
  \pgftext[at={\pgfpointxy{650.0}{650.0}}]{$x$}
  \pgftext[at={\pgfpointxy{800.0}{650.0}}]{$z$}
  \pgftext[at={\pgfpointxy{325.0}{500.0}}]{$a$}
  \pgftext[at={\pgfpointxy{400.0}{500.0}}]{$b$}
  \pgftext[at={\pgfpointxy{475.0}{500.0}}]{$c$}
  \pgftext[bottom,at={\pgfpointxy{1125.0}{637.0}}]{$\calD$}
  \pgftext[top,at={\pgfpointxy{725.0}{313.0}}]{$\calB$}
  \pgftext[bottom,at={\pgfpointxy{725.0}{662.0}}]{$\calC$}
  \pgftext[bottom,at={\pgfpointxy{400.0}{537.0}}]{$\calA$}
\end{pgfpicture}
  \caption{A counterexample for the amalgamation property for the class of all at most 2-dimensional finite linear orders}
  \label{mnmrf-c.fig.no-amalg}
\end{figure}

The class of all at most 2-dimensional finite linear orders is not a \Fraisse\ age. Namely,
it is a well-known fact that one of the fundamental properties of a \Fraisse\ age is the
\emph{amalgamation property (AP)}, which is a requirement that
every span $\calB \overset f \hookleftarrow \calA \overset g \hookrightarrow \calC$ with $\calA, \calB, \calC \in \KK$
completes to a square with $\calD \in \KK$:
\begin{center}
  \begin{tikzcd}
    \calC \arrow[r, hookrightarrow, "g'"] & \calD\\
    \calA \arrow[r, hookrightarrow, "f"'] \arrow[u, hookrightarrow, "g"] & \calB \arrow[u, hookrightarrow, "f'"']
  \end{tikzcd}
\end{center}
This property fails for the class of all at most 2-dimensional finite linear orders: the span
$\calB \hookleftarrow \calA \hookrightarrow \calC$ in Fig.~\ref{mnmrf-c.fig.no-amalg}
can easily be completed to a square in the class of \emph{all} partially ordered sets;
however, every candidate for $\calD$ embeds a crown on six elements, so the dimension of each such $\calD$ is at least~3.

Nevertheless, we shall prove that the reduct of the generic permutation we are interested in is the \emph{weak \Fraisse\ limit} of this class, and hence a generic structure.

A class $\KK$ of finite relational structures over the same relational language has the \emph{joint embedding property (JEP)}
if every pair of structures from $\KK$ embeds into a structure from $\KK$; and it has the \emph{weak
amalgamation property (WAP)} if for every $\calA \in \KK$ there is a $\overline\calA \in \KK$ and
an embedding $e : \calA \hookrightarrow \overline\calA$ such that for every choice of $\calB, \calC \in \KK$
and embeddings $f : \overline\calA \hookrightarrow \calB$ and $g : \overline\calA \hookrightarrow \calC$
there exist a $\calD \in \KK$ and embeddings $f' : \calB \hookrightarrow \calD$ and
$g' : \calC \hookrightarrow \calD$ satisfying $f' \circ f \circ e = g' \circ g \circ e$:
\begin{center}
  \begin{tikzcd}
                          & \calC \arrow[r, hookrightarrow, "g'"] & \calD\\
    \calA \arrow[r, hookrightarrow, "e"'] & \overline \calA \arrow[r, hookrightarrow, "f"'] \arrow[u, hookrightarrow, "g"] & \calB \arrow[u, hookrightarrow, "f'"']
  \end{tikzcd}
\end{center}

The importance of the weak amalgamation property was first recognized by Ivanov \cite{ivanov-1999}
and later independently by Kechris and Rosendal \cite{kechris-rosendal-2007} in their study of generic automorphisms.
The importance of the weak amalgamation property with respect to the existence of generic structures was
recognized independently by Kubi\'s \cite{kubis-2022}, and Panagiotopoulos and Tent \cite{panag-tent-2022}.
Namely, if $\KK$ is an age of a countable relational structure with (JEP) and (WAP), then there is a
unique (up to isomorphism) countable structure $\Omega$ on $\omega = \{0, 1, 2, \ldots\}$ with $\Age(\Omega) = \KK$
whose isomorphism class is comeager in a naturally defined Baire space of all structures on $\omega$ whose age is $\KK$,
(see~\cite{panag-tent-2022} for details). We call $\Omega$ the \emph{weak \Fraisse\ limit of $\KK$}
and say that $\Omega$ is \emph{generic} among countable structures whose age is $\KK$.

In order to verify the weak amalgamation property,
it suffices to show that for every structure $\calA$ there is a well-behaved structure $\overline\calA$ into which
$\calA$ embeds, and over which we can weakly amalgamate. As we shall se below, in the case of finite posets of dimension $\le 2$
the well-behaved structures $\overline\calA$ are of the form $(A, \Boxed{\le^A}) \times (B, \Boxed{\le^B})$
where $(A, \Boxed{\le^A})$ and $(B, \Boxed{\le^B})$ are finite linear orders. 
We shall now introduce a bit of notation and establish several elementary facts about posets of this form,
which will be used in the main argument below.

Let $(A, \Boxed{\le^A})$ and $(B, \Boxed{\le^B})$ be finite linear orders. Then $(A, \Boxed{\le^A}) \times (B, \Boxed{\le^B})$
is a finite partially ordered set $(A \times B, \Boxed{\sqsubseteq^{A \times B}})$ where
$$
  (a_1, b_1) \mathrel{\sqsubseteq^{A \times B}} (a_2, b_2) \text{ iff } a_1 \le^A a_2 \land b_1 \le^B b_2.
$$
Let $\lex^{A \times B}$ and $\alex^{A \times B}$ denote the \emph{lexicographic} and \emph{antilexicographic} order
on $A \times B$, respectively, which are linear orders defined as follows:
\begin{align*}
  (a_1, b_1) \mathrel{\lex^{A \times B}} (a_2, b_2) \text{ iff }
  & a_1 <^A a_2 \lor (a_1 = a_2 \land b_1 \le^B b_2), \text{ and}\\
  (a_1, b_1) \mathrel{\alex^{A \times B}} (a_2, b_2) \text{ iff }
  & b_1 <^B b_2 \lor (b_1 = b_2 \land a_1 \le^A a_2).
\end{align*}
Then it is easy to show that:
$$
  \Boxed{\sqsubseteq^{A \times B}} = (\Boxed{\lex^{A \times B}}) \cap (\Boxed{\alex^{A \times B}}).
$$

\begin{LEMA}
  Let $(A, \Boxed{\le^A})$ and $(B, \Boxed{\le^B})$ be nonempty finite linear orders, and let
  $(A \times B, \Boxed{\sqsubseteq^{A \times B}}) = (A, \Boxed{\le^A}) \times (B, \Boxed{\le^B})$.
  Let $\le_1^{A \times B}$ and $\le_2^{A \times B}$ be linear orders on $A \times B$ such that
  $\Boxed{\sqsubseteq^{A \times B}} = (\Boxed{\le_1^{A \times B}}) \cap (\Boxed{\le_2^{A \times B}})$. Then
  $\{\Boxed{\le_1^{A \times B}}, \Boxed{\le_2^{A \times B}}\} = \{\Boxed{\lex^{A \times B}}, \Boxed{\alex^{A \times B}}\}$.
\end{LEMA}
\begin{proof}
  The proof is by induction on $|B|$. If $|B| = 1$ then $(A \times B, \Boxed{\sqsubseteq^{A \times B}}) \cong (A, \Boxed{\le^A})$
  and:
  $$
    (\le_1^{A \times B}) = (\le_2^{A \times B}) = (\lex^{A \times B}) = (\alex^{A \times B}).
  $$
  For the induction step, assume that the statement is true whenever $|B| = n - 1$ and let $(B, \Boxed{\le^B})$ be an $n$-element
  linearly ordered set, say $B = \{1, 2, \ldots, n\}$ with the usual ordering of the integers, where $n \ge 2$.
  For the sake of simplicity, let $A = \{a, b, c, \ldots, z\}$ so that
  $$
    A \times B = \{a_1, b_1, \ldots, z_1, \, a_2, b_2, \ldots, z_2, \, \ldots, \, a_n, b_n, \ldots, z_n\},
  $$
  see Fig.~\ref{corrigendum.fig.AxB}.

  \begin{figure}
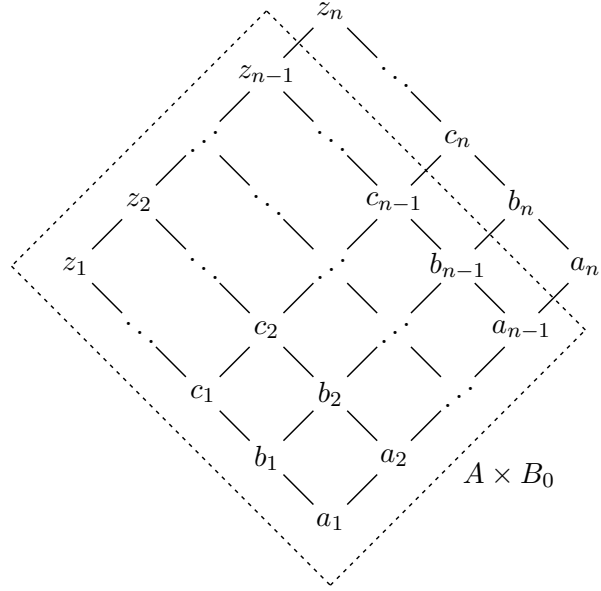

    \centering
\begin{pgfpicture}
  \pgfsetxvec{\pgfpoint{\acadpgfunit}{0pt}}
  \pgfsetyvec{\pgfpoint{0pt}{\acadpgfunit}}
  \pgfsetlinewidth{\acadpgflinewidth}
  \pgftransformshift{\pgfpointxy{50.0}{-112.5}}

  \begin{pgfscope}
    \pgfpathmoveto{\pgfpointxy{126.517}{676.517}}
    \pgfpathlineto{\pgfpointxy{173.483}{723.483}}
    \pgfusepath{stroke}
  \end{pgfscope}
  \begin{pgfscope}
    \pgfpathmoveto{\pgfpointxy{226.517}{776.517}}
    \pgfpathlineto{\pgfpointxy{273.483}{823.483}}
    \pgfusepath{stroke}
  \end{pgfscope}
  \begin{pgfscope}
    \pgfpathmoveto{\pgfpointxy{326.517}{876.517}}
    \pgfpathlineto{\pgfpointxy{373.483}{923.483}}
    \pgfusepath{stroke}
  \end{pgfscope}
  \begin{pgfscope}
    \pgfpathmoveto{\pgfpointxy{426.517}{976.517}}
    \pgfpathlineto{\pgfpointxy{473.483}{1023.48}}
    \pgfusepath{stroke}
  \end{pgfscope}
  \begin{pgfscope}
    \pgfpathmoveto{\pgfpointxy{526.517}{1023.48}}
    \pgfpathlineto{\pgfpointxy{573.483}{976.517}}
    \pgfusepath{stroke}
  \end{pgfscope}
  \begin{pgfscope}
    \pgfpathmoveto{\pgfpointxy{226.517}{523.483}}
    \pgfpathlineto{\pgfpointxy{273.483}{476.517}}
    \pgfusepath{stroke}
  \end{pgfscope}
  \begin{pgfscope}
    \pgfpathmoveto{\pgfpointxy{326.517}{476.517}}
    \pgfpathlineto{\pgfpointxy{373.483}{523.483}}
    \pgfusepath{stroke}
  \end{pgfscope}
  \begin{pgfscope}
    \pgfpathmoveto{\pgfpointxy{426.517}{576.517}}
    \pgfpathlineto{\pgfpointxy{473.483}{623.483}}
    \pgfusepath{stroke}
  \end{pgfscope}
  \begin{pgfscope}
    \pgfpathmoveto{\pgfpointxy{526.517}{676.517}}
    \pgfpathlineto{\pgfpointxy{573.483}{723.483}}
    \pgfusepath{stroke}
  \end{pgfscope}
  \begin{pgfscope}
    \pgfpathmoveto{\pgfpointxy{626.517}{776.517}}
    \pgfpathlineto{\pgfpointxy{673.483}{823.483}}
    \pgfusepath{stroke}
  \end{pgfscope}
  \begin{pgfscope}
    \pgfpathmoveto{\pgfpointxy{726.517}{823.483}}
    \pgfpathlineto{\pgfpointxy{773.483}{776.517}}
    \pgfusepath{stroke}
  \end{pgfscope}
  \begin{pgfscope}
    \pgfpathmoveto{\pgfpointxy{773.483}{723.483}}
    \pgfpathlineto{\pgfpointxy{726.517}{676.517}}
    \pgfusepath{stroke}
  \end{pgfscope}
  \begin{pgfscope}
    \pgfpathmoveto{\pgfpointxy{673.483}{623.483}}
    \pgfpathlineto{\pgfpointxy{626.517}{576.517}}
    \pgfusepath{stroke}
  \end{pgfscope}
  \begin{pgfscope}
    \pgfpathmoveto{\pgfpointxy{573.483}{523.483}}
    \pgfpathlineto{\pgfpointxy{526.517}{476.517}}
    \pgfusepath{stroke}
  \end{pgfscope}
  \begin{pgfscope}
    \pgfpathmoveto{\pgfpointxy{473.483}{423.483}}
    \pgfpathlineto{\pgfpointxy{426.517}{376.517}}
    \pgfusepath{stroke}
  \end{pgfscope}
  \begin{pgfscope}
    \pgfpathmoveto{\pgfpointxy{426.517}{323.483}}
    \pgfpathlineto{\pgfpointxy{473.483}{276.517}}
    \pgfusepath{stroke}
  \end{pgfscope}
  \begin{pgfscope}
    \pgfpathmoveto{\pgfpointxy{526.517}{276.517}}
    \pgfpathlineto{\pgfpointxy{573.483}{323.483}}
    \pgfusepath{stroke}
  \end{pgfscope}
  \begin{pgfscope}
    \pgfpathmoveto{\pgfpointxy{626.517}{376.517}}
    \pgfpathlineto{\pgfpointxy{673.483}{423.483}}
    \pgfusepath{stroke}
  \end{pgfscope}
  \begin{pgfscope}
    \pgfpathmoveto{\pgfpointxy{726.517}{476.517}}
    \pgfpathlineto{\pgfpointxy{773.483}{523.483}}
    \pgfusepath{stroke}
  \end{pgfscope}
  \begin{pgfscope}
    \pgfpathmoveto{\pgfpointxy{826.517}{576.517}}
    \pgfpathlineto{\pgfpointxy{873.483}{623.483}}
    \pgfusepath{stroke}
  \end{pgfscope}
  \begin{pgfscope}
    \pgfpathmoveto{\pgfpointxy{873.483}{676.517}}
    \pgfpathlineto{\pgfpointxy{826.517}{723.483}}
    \pgfusepath{stroke}
  \end{pgfscope}
  \begin{pgfscope}
    \pgfpathmoveto{\pgfpointxy{673.483}{876.517}}
    \pgfpathlineto{\pgfpointxy{626.517}{923.483}}
    \pgfusepath{stroke}
  \end{pgfscope}
  \begin{pgfscope}
    \pgfpathmoveto{\pgfpointxy{373.483}{376.517}}
    \pgfpathlineto{\pgfpointxy{326.517}{423.483}}
    \pgfusepath{stroke}
  \end{pgfscope}
  \begin{pgfscope}
    \pgfpathmoveto{\pgfpointxy{173.483}{576.517}}
    \pgfpathlineto{\pgfpointxy{126.517}{623.483}}
    \pgfusepath{stroke}
  \end{pgfscope}
  \begin{pgfscope}
    \pgfpathmoveto{\pgfpointxy{226.517}{723.483}}
    \pgfpathlineto{\pgfpointxy{273.483}{676.517}}
    \pgfusepath{stroke}
  \end{pgfscope}
  \begin{pgfscope}
    \pgfpathmoveto{\pgfpointxy{326.517}{623.483}}
    \pgfpathlineto{\pgfpointxy{373.483}{576.517}}
    \pgfusepath{stroke}
  \end{pgfscope}
  \begin{pgfscope}
    \pgfpathmoveto{\pgfpointxy{426.517}{523.483}}
    \pgfpathlineto{\pgfpointxy{473.483}{476.517}}
    \pgfusepath{stroke}
  \end{pgfscope}
  \begin{pgfscope}
    \pgfpathmoveto{\pgfpointxy{526.517}{423.483}}
    \pgfpathlineto{\pgfpointxy{573.483}{376.517}}
    \pgfusepath{stroke}
  \end{pgfscope}
  \begin{pgfscope}
    \pgfpathmoveto{\pgfpointxy{326.517}{823.483}}
    \pgfpathlineto{\pgfpointxy{373.483}{776.517}}
    \pgfusepath{stroke}
  \end{pgfscope}
  \begin{pgfscope}
    \pgfpathmoveto{\pgfpointxy{426.517}{723.483}}
    \pgfpathlineto{\pgfpointxy{473.483}{676.517}}
    \pgfusepath{stroke}
  \end{pgfscope}
  \begin{pgfscope}
    \pgfpathmoveto{\pgfpointxy{526.517}{623.483}}
    \pgfpathlineto{\pgfpointxy{573.483}{576.517}}
    \pgfusepath{stroke}
  \end{pgfscope}
  \begin{pgfscope}
    \pgfpathmoveto{\pgfpointxy{626.517}{523.483}}
    \pgfpathlineto{\pgfpointxy{673.483}{476.517}}
    \pgfusepath{stroke}
  \end{pgfscope}
  \begin{pgfscope}
    \pgfpathmoveto{\pgfpointxy{773.483}{576.517}}
    \pgfpathlineto{\pgfpointxy{726.517}{623.483}}
    \pgfusepath{stroke}
  \end{pgfscope}
  \begin{pgfscope}
    \pgfpathmoveto{\pgfpointxy{673.483}{676.517}}
    \pgfpathlineto{\pgfpointxy{626.517}{723.483}}
    \pgfusepath{stroke}
  \end{pgfscope}
  \begin{pgfscope}
    \pgfpathmoveto{\pgfpointxy{573.483}{776.517}}
    \pgfpathlineto{\pgfpointxy{526.517}{823.483}}
    \pgfusepath{stroke}
  \end{pgfscope}
  \begin{pgfscope}
    \pgfpathmoveto{\pgfpointxy{473.483}{876.517}}
    \pgfpathlineto{\pgfpointxy{426.517}{923.483}}
    \pgfusepath{stroke}
  \end{pgfscope}
  \begin{pgfscope}
    \pgfsetdash{{1.5pt}{2pt}}{0pt}
    \pgfpathmoveto{\pgfpointxy{900.0}{550.0}}
    \pgfpathlineto{\pgfpointxy{400.0}{1050.0}}
    \pgfusepath{stroke}
  \end{pgfscope}
  \begin{pgfscope}
    \pgfsetdash{{1.5pt}{2pt}}{0pt}
    \pgfpathmoveto{\pgfpointxy{400.0}{1050.0}}
    \pgfpathlineto{\pgfpointxy{0.0}{650.0}}
    \pgfusepath{stroke}
  \end{pgfscope}
  \begin{pgfscope}
    \pgfsetdash{{1.5pt}{2pt}}{0pt}
    \pgfpathmoveto{\pgfpointxy{0.0}{650.0}}
    \pgfpathlineto{\pgfpointxy{500.0}{150.0}}
    \pgfusepath{stroke}
  \end{pgfscope}
  \begin{pgfscope}
    \pgfsetdash{{1.5pt}{2pt}}{0pt}
    \pgfpathmoveto{\pgfpointxy{500.0}{150.0}}
    \pgfpathlineto{\pgfpointxy{900.0}{550.0}}
    \pgfusepath{stroke}
  \end{pgfscope}
  \pgftext[at={\pgfpointxy{500.0}{250.0}}]{$a_1$}
  \pgftext[at={\pgfpointxy{400.0}{350.0}}]{$b_1$}
  \pgftext[at={\pgfpointxy{300.0}{450.0}}]{$c_1$}
  \pgftext[at={\pgfpointxy{100.0}{650.0}}]{$z_1$}
  \pgftext[at={\pgfpointxy{201.525}{559.913}}]{$\ddots$}
  \pgftext[at={\pgfpointxy{600.0}{350.0}}]{$a_2$}
  \pgftext[at={\pgfpointxy{500.0}{450.0}}]{$b_2$}
  \pgftext[at={\pgfpointxy{400.0}{550.0}}]{$c_2$}
  \pgftext[at={\pgfpointxy{200.0}{750.0}}]{$z_2$}
  \pgftext[at={\pgfpointxy{301.525}{659.913}}]{$\ddots$}
  \pgftext[at={\pgfpointxy{700.763}{459.151}}]{$\iddots$}
  \pgftext[at={\pgfpointxy{600.763}{559.151}}]{$\iddots$}
  \pgftext[at={\pgfpointxy{500.763}{659.151}}]{$\iddots$}
  \pgftext[at={\pgfpointxy{300.763}{859.151}}]{$\iddots$}
  \pgftext[at={\pgfpointxy{800.0}{550.0}}]{$a_{n-1}$}
  \pgftext[at={\pgfpointxy{700.0}{650.0}}]{$b_{n-1}$}
  \pgftext[at={\pgfpointxy{600.0}{750.0}}]{$c_{n-1}$}
  \pgftext[at={\pgfpointxy{400.0}{950.0}}]{$z_{n-1}$}
  \pgftext[at={\pgfpointxy{501.525}{859.913}}]{$\ddots$}
  \pgftext[at={\pgfpointxy{900.0}{650.0}}]{$a_n$}
  \pgftext[at={\pgfpointxy{800.0}{750.0}}]{$b_n$}
  \pgftext[at={\pgfpointxy{700.0}{850.0}}]{$c_n$}
  \pgftext[at={\pgfpointxy{500.0}{1050.0}}]{$z_n$}
  \pgftext[at={\pgfpointxy{601.525}{959.913}}]{$\ddots$}
  \pgftext[at={\pgfpointxy{401.525}{759.913}}]{$\ddots$}
  \pgftext[top,left,at={\pgfpointxy{708.0}{342.0}}]{$A \times B_0$}
\end{pgfpicture}
	\caption{$(A \times B, \Boxed{\sqsubseteq^{A \times B}})$ in the proof of Lemma A}
	\label{corrigendum.fig.AxB}
  \end{figure}

  Let $\le_1^{A \times B}$ and $\le_2^{A \times B}$ be linear orders on $A \times B$ such that
  $$
    \Boxed{\sqsubseteq^{A \times B}} = (\Boxed{\le_1^{A \times B}}) \cap (\Boxed{\le_2^{A \times B}}).
  $$
  Put $B_0 = \{1, 2, \ldots, n-1\}$ and let $\le_1^{A \times B_0}$ and $\le_2^{A \times B_0}$ be the restrictions of
  $\le_1^{A \times B}$ and $\le_2^{A \times B}$, respectively, to $A \times B_0$. Then 
  $$
    \Boxed{\sqsubseteq^{A \times B_0}} = (\Boxed{\le_1^{A \times B_0}}) \cap (\Boxed{\le_2^{A \times B_0}}),
  $$
  so, by the induction hypothesis, 
  $$
    \{\Boxed{\le_1^{A \times B_0}}, \Boxed{\le_2^{A \times B_0}}\} = \{\Boxed{\lex^{A \times B_0}}, \Boxed{\alex^{A \times B_0}}\}.
  $$
  Without loss of generality we can assume that:
  \begin{align*}
    \Boxed{\le_1^{A \times B_0}} = \Boxed{\alex^{A \times B_0}} = a_1 b_1 \ldots z_1 \; a_2 b_2 \ldots z_2 \; \ldots \; a_{n-1} b_{n-1} \ldots z_{n-1};\\
    \Boxed{\le_2^{A \times B_0}} = \Boxed{\lex^{A \times B_0}} = a_1 a_2 \ldots a_{n-1} \; b_1 b_2 \ldots b_{n-1} \; \ldots \; z_1 z_2 \ldots z_{n-1}.
  \end{align*}
  Let us show that there is exactly one way to reconstruct $\Boxed{\le_1^{A \times B}}$ and $\Boxed{\le_2^{A \times B}}$
  from $\Boxed{\le_1^{A \times B_0}}$ and $\Boxed{\le_2^{A \times B_0}}$.
  
  Let us start with $a_n$. Since $a_{n-1} \mathrel{\sqsubseteq^{A \times B}} a_n$ it follows that
  $a_{n-1} \mathrel{\le_1^{A \times B}} a_n$ and $a_{n-1} \mathrel{\le_2^{A \times B}} a_n$.
  By the induction hypothesis we have that $b_1 \mathrel{\le_1^{A \times B_0}} a_{n-1}$, so
  $b_1 \mathrel{\le_1^{A \times B}} a_{n}$. Then $a_n \mathrel{\le_2^{A \times B}} b_1$ because
  $a_n$ and $b_1$ are incomparable with respect to $\Boxed{\sqsubseteq^{A \times B}} = (\Boxed{\le_1^{A \times B}}) \cap (\Boxed{\le_2^{A \times B}})$.
  Therefore,
  $$
    a_{n-1} \mathrel{\le_2^{A \times B}} a_n \mathrel{\le_2^{A \times B}} b_1,
  $$
  whence follows that $\Boxed{\le_2^{A \times B}}$ contains the following linear order:
  $$
    \Boxed{\le_2^{A \times B}}: \; a_1 a_2 \ldots a_{n-1} \underset\uparrow{a_n} \; b_1 b_2 \ldots b_{n-1} \; \ldots \; z_1 z_2 \ldots z_{n-1}
  $$
  Consequently, $a_n \mathrel{\le_2^{A \times B}} z_{n-1}$, so $z_{n-1} \mathrel{\le_1^{A \times B}} a_n$ because
  $a_n$ and $z_{n-1}$ are incomparable with respect to $\Boxed{\sqsubseteq^{A \times B}}$. This implies that
  $\Boxed{\le_1^{A \times B}}$ contains the following linear order:
  $$
    \Boxed{\le_1^{A \times B}}: \; a_1 b_1 \ldots z_1 \; a_2 b_2 \ldots z_2 \; \ldots \; a_{n-1} b_{n-1} \ldots z_{n-1} \; a_n.
  $$
  But, $a_n \Boxed{\sqsubseteq^{A \times B}} b_n \Boxed{\sqsubseteq^{A \times B}} \ldots \Boxed{\sqsubseteq^{A \times B}} z_n$, so
  we can fully describe $\Boxed{\le_1^{A \times B}}$ as:
  $$
    \Boxed{\le_1^{A \times B}} = a_1 b_1 \ldots z_1 \; a_2 b_2 \ldots z_2 \; \ldots \; a_{n-1} b_{n-1} \ldots z_{n-1} \; a_n b_n \ldots z_n = \Boxed{\alex^{A \times B}}.
  $$

  Let us complete the proof that $\Boxed{\le_2^{A \times B}} = \Boxed{\lex^{A \times B}}$. We have already placed $a_n$, so let
  us now move on to $b_n$. Since $b_{n-1} \mathrel{\sqsubseteq^{A \times B}} b_n$ it follows that
  $b_{n-1} \mathrel{\le_1^{A \times B}} b_n$ and $b_{n-1} \mathrel{\le_2^{A \times B}} b_n$.
  By the induction hypothesis we have that $c_1 \mathrel{\le_1^{A \times B_0}} b_{n-1}$, so
  $c_1 \mathrel{\le_1^{A \times B}} b_{n}$. Then $b_n \mathrel{\le_2^{A \times B}} c_1$ because
  $b_n$ and $c_1$ are incomparable with respect to $\Boxed{\sqsubseteq^{A \times B}}$. Therefore,
  $$
    b_{n-1} \mathrel{\le_2^{A \times B}} b_n \mathrel{\le_2^{A \times B}} c_1,
  $$
  whence follows that $\Boxed{\le_2^{A \times B}}$ contains the following linear order:
  $$
    \Boxed{\le_2^{A \times B}}: \; a_1 a_2 \ldots a_{n-1} a_n \; b_1 b_2 \ldots b_{n-1} \underset\uparrow{b_n} \; c_1 \ldots \; z_1 z_2 \ldots z_{n-1}.
  $$
  We can now repeat the argument to place $c_n$ immediately after $c_{n-1}$, \ldots, $y_n$ immediately after $y_{n-1}$
  showing, thus, that $\Boxed{\le_2^{A \times B}}$ contains the following linear order:
  $$
    \Boxed{\le_2^{A \times B}}: \; a_1 a_2 \ldots a_n \; b_1 b_2 \ldots b_n \; c_1 c_2 \ldots c_n \; \ldots \; 
	y_1 y_2 \ldots y_n \; z_1 z_2 \ldots z_{n-1}.
  $$
  The fact that $z_n$ is the largest element in $(A \times B, \Boxed{\sqsubseteq^{A \times B}})$ places $z_n$ at the very end of the above list,
  showing that $\Boxed{\le_2^{A \times B}} = \Boxed{\lex^{A \times B}}$.
\end{proof}

\begin{LEMB}
  Let $(A, \Boxed{\le^A})$, $(B, \Boxed{\le^B})$, $(C, \Boxed{\le^C})$ and $(D, \Boxed{\le^D})$
  be nonempty finite linear orders, and let $f : (A, \Boxed{\le^A}) \times (B, \Boxed{\le^B}) \hookrightarrow
  (C, \Boxed{\le^C}) \times (D, \Boxed{\le^D})$ be an embedding. Then
  \begin{itemize}
  \item $f : (A \times B, \Boxed{\lex^{A \times B}}, \Boxed{\alex^{A \times B}}) \hookrightarrow (C \times D, \Boxed{\lex^{C \times D}}, \Boxed{\alex^{C \times D}})$ is an embedding, or
  \item $f : (A \times B, \Boxed{\lex^{A \times B}}, \Boxed{\alex^{A \times B}}) \hookrightarrow (C \times D, \Boxed{\alex^{C \times D}}, \Boxed{\lex^{C \times D}})$ is an embedding.
  \end{itemize}
\end{LEMB}
\begin{proof}
  Let $f : (A, \Boxed{\le^A}) \times (B, \Boxed{\le^B}) \hookrightarrow
  (C, \Boxed{\le^C}) \times (D, \Boxed{\le^D})$ be an embedding. In other words,
  $f : (A \times B, \Boxed{\sqsubseteq^{A \times B}}) \hookrightarrow (C \times D, \Boxed{\sqsubseteq^{C \times D}})$ is an embedding.
  Define linear orders $\le_1^{A \times B}$ and $\le_2^{A \times B}$ on $A \times B$ as follows:
  \begin{align*}
    (a_1, b_1) \mathrel{\le_1^{A \times B}} (a_2, b_2) &\text{ iff } f(a_1, b_1) \mathrel{\lex^{C \times D}} f(a_2, b_2), \text{ and}\\
    (a_1, b_1) \mathrel{\le_2^{A \times B}} (a_2, b_2) &\text{ iff } f(a_1, b_1) \mathrel{\alex^{C \times D}} f(a_2, b_2).
  \end{align*}
  Then $f : (A \times B, \Boxed{\le_1^{A \times B}}, \Boxed{\le_2^{A \times B}}) \hookrightarrow (C \times D, \Boxed{\lex^{C \times D}}, \Boxed{\alex^{C \times D}})$ is an embedding.
  On the other hand, it is easy to check that:
  $$
    \Boxed{\sqsubseteq^{A \times B}} = (\Boxed{\le_1^{A \times B}}) \cap (\Boxed{\le_2^{A \times B}}),
  $$
  so, by Lemma~A, it must be the case that
  $\{\Boxed{\le_1^{A \times B}}, \Boxed{\le_2^{A \times B}}\} = \{\Boxed{\lex^{A \times B}}, \Boxed{\alex^{A \times B}}\}$.
\end{proof}

Following \cite{kubis-2022}, in order to show that a countable structure $\Omega$ is a weak \Fraisse\ limit
of an age $\KK$, and hence a generic structure, it suffices to show that $\Omega$ embeds every structure from $\KK$,
and that it is \emph{weakly injective} in the following sense: for every $\calA \in \KK$ and every embedding
$f : \calA \hookrightarrow \Omega$ there is an embedding $e : \calA \hookrightarrow \overline\calA$ with $\overline \calA \in \KK$
such that for every embedding $g : \overline\calA \hookrightarrow \calB$ with $\calB \in \KK$ there is an
embedding $h : \calB \hookrightarrow \Omega$ such that $h \circ g \circ e = f$:
\begin{center}
  \begin{tikzcd}
    \calA \arrow[d, hookrightarrow, "f"'] \arrow[r, hookrightarrow, "e"] & \overline\calA \arrow[r, hookrightarrow, "g"] & \calB \arrow[dll, hookrightarrow, dashed, "h"]\\
	\Omega
  \end{tikzcd}
\end{center}

Instead of the strict version $\calQ = (\QQ, \Boxed{\prec_1}, \Boxed{\prec_2})$ of the generic permutation,
for the constructions that follow it will be more convenient to work with the reflexive version
$\calQ_\refl = (\QQ, \Boxed{\preceq_1}, \Boxed{\preceq_2})$. Define $\preceq$ on $\QQ$ as follows:
$$
  a \preceq b \text{ iff } a \preceq_1 b \land a \preceq_2 b.
$$
Note that $\preceq_1$ and $\preceq_2$ are linear orders, while $\preceq$ is a 2-dimensional partial order on $\QQ$.
So, $(\QQ, \Boxed\preceq)$ is a 2-dimensional poset and Theorem~6.4 shows that this poset has big Ramsey degrees.

\begin{THMC}
  $(\QQ, \Boxed\preceq)$ is the generic 2-dimensional poset.
\end{THMC}
\begin{proof}
  As we have argued above, it suffices to show that $(\QQ, \Boxed\preceq)$ is a weak \Fraisse\ limit of the class $\KK_{\le2}$
  of all finite at most 2-dimensional posets.
  
  \bigskip
  
  Claim 1. $\Age(\QQ, \Boxed\preceq) = \KK_{\le2}$.
  
  Proof. $(\subseteq)$ Let $(A, \Boxed{\preceq^A})$ be a finite poset and $f : (A, \Boxed{\preceq^A}) \hookrightarrow (\QQ, \Boxed\preceq)$
  an embedding. Define linear orders $\le^A_1$ and $\le^A_2$ on $A$ as follows:
  $$
    a \mathrel{\le_i^A} b \text{ iff } f(a) \mathrel{\preceq_i} f(b) \text{ in } \calQ_\refl, \quad i \in \{1, 2\}.
  $$
  Then it is easy to check that $\Boxed{\preceq^A} = (\Boxed{\le_1^A}) \cap (\Boxed{\le_2^A})$ proving that
  $(A, \Boxed{\preceq^A}) \in \KK_{\le 2}$.
  
  $(\supseteq)$ Take any $(A, \Boxed{\preceq^A}) \in \KK_{\le 2}$ and let $\le^A_1$ and $\le^A_2$ be linear orders on $A$
  such that $\Boxed{\preceq^A} = (\Boxed{\le_1^A}) \cap (\Boxed{\le_2^A})$. Then $(A, \Boxed{\le_1^A}, \Boxed{\le_2^A})$
  is a finite permutation, so there is an embedding $f : (A, \Boxed{\le_1^A}, \Boxed{\le_2^A}) \hookrightarrow (\QQ, \Boxed{\preceq_1}, \Boxed{\preceq_2})$
  because $(\QQ, \Boxed{\preceq_1}, \Boxed{\preceq_2})$ is the \Fraisse\ limit of the class of all finite permutations.
  Then it is easy to check that $f$ is also an embedding $f : (A, \Boxed{\preceq^A}) \hookrightarrow (\QQ, \Boxed\preceq)$,
  so $(A, \Boxed{\preceq^A}) \in \Age(\QQ, \Boxed\preceq)$.
  
  \bigskip
  
  Claim 2. $(\QQ, \Boxed\preceq)$ is weakly injective for the class $\KK_{\le2}$.
  
  Proof. Take any $(A, \Boxed{\preceq^A}) \in K_{\le 2}$ and any embedding $f : (A, \Boxed{\preceq^A}) \hookrightarrow (\QQ, \Boxed{\preceq})$.
  Define linear orders $\le_1^A$ and $\le_2^A$ on $A$ as follows:
  $$
    a \mathrel{\le_i^A} b \text{ iff } f(a) \mathrel{\preceq_i} f(b) \text{ in } \calQ_\refl, \quad i \in \{1, 2\}.
  $$
  Then $\Boxed{\preceq^A} = (\Boxed{\le_1^A}) \cap (\Boxed{\le_2^A})$, so
  $e : (A, \Boxed{\preceq^A}) \to (A, \Boxed{\le_1^A}) \times (A, \Boxed{\le_2^A})$ given by $e(a) = (a, a)$ is an embedding.
  
  Now, take any $(B, \Boxed{\preceq^B}) \in \KK_{\le 2}$ and any embedding
  $g : (A, \Boxed{\le_1^A}) \times (A, \Boxed{\le_2^A}) \hookrightarrow (B, \Boxed{\preceq^B})$.
  Since the dimension of $(B, \Boxed{\preceq^B})$ is $\le 2$, there are finite linear orders $\le^B_1$ and $\le^B_2$ on $B$
  such that $\Boxed{\preceq^B} = (\Boxed{\le_1^B}) \cap (\Boxed{\le_2^B})$ and
  $\ell : (B, \Boxed{\preceq^B}) \to (B, \Boxed{\le_1^B}) \times (B, \Boxed{\le_2^B})$ given by $\ell(b) = (b,b)$ is an embedding.
  Then
  $$
    \ell \circ g : (A, \Boxed{\le_1^A}) \times (A, \Boxed{\le_2^A}) \hookrightarrow (B, \Boxed{\le_1^B}) \times (B, \Boxed{\le_2^B})
  $$
  is an embedding, so by Lemma~B one of the two maps bellow is an embedding:
  \begin{itemize}
  \item $\ell \circ g : (A \times A, \Boxed{\lex^{A \times A}}, \Boxed{\alex^{A \times A}}) \to (B \times B, \Boxed{\lex^{B \times B}}, \Boxed{\alex^{B \times B}})$ or
  \item $\ell \circ g : (A \times A, \Boxed{\lex^{A \times A}}, \Boxed{\alex^{A \times A}}) \to (B \times B, \Boxed{\alex^{B \times B}}, \Boxed{\lex^{B \times B}})$.
  \end{itemize}
  Without loss of generality we can assume that $\ell \circ g : (A \times A, \Boxed{\lex^{A \times A}}, \Boxed{\alex^{A \times A}})
  \hookrightarrow (B \times B, \Boxed{\alex^{B \times B}}, \Boxed{\lex^{B \times B}})$ is an embedding.
  
  To conclude the proof, note that $e$ given as above is also an embedding
  $e : (A, \Boxed{\le_1^A}, \Boxed{\le_2^A}) \hookrightarrow (A \times A, \Boxed{\lex^{A \times A}}, \Boxed{\alex^{A \times A}})$,
  and that, by construction of $\Boxed{\le_1^A}$ and $\Boxed{\le_2^A}$, $f$ is also an embedding
  $f : (A, \Boxed{\le_1^A}, \Boxed{\le_2^A}) \hookrightarrow (\QQ, \Boxed{\preceq_1}, \Boxed{\preceq_2})$. Since
  $(\QQ, \Boxed{\preceq_1}, \Boxed{\preceq_2})$ is the \Fraisse\ limit of the class of all finite permutations, there is
  an embedding $k : (B \times B, \Boxed{\alex^{B \times B}}, \Boxed{\lex^{B \times B}}) \hookrightarrow (\QQ, \Boxed{\preceq_1}, \Boxed{\preceq_2})$
  such that $k \circ (\ell \circ g \circ e) = f$:
  \begin{center}
    \begin{tikzcd}[column sep=large]
	  (A, \Boxed{\le_1^A}, \Boxed{\le_2^A}) \arrow[d, hookrightarrow, "f"'] \arrow[r, hookrightarrow, "\ell \circ g \circ e"] & (B \times B, \Boxed{\alex^{B \times B}}, \Boxed{\lex^{B \times B}}) \arrow[dl, hookrightarrow, dashed, "k"]\\
	  (\QQ, \Boxed{\preceq_1}, \Boxed{\preceq_2})
	\end{tikzcd}
  \end{center}
  Again, it is easy to verify that $k$ is also an embedding
  $$
    k : (B \times B, \Boxed{\alex^{B \times B}} \cap \Boxed{\lex^{B \times B}}) \hookrightarrow (\QQ, \Boxed{\preceq}),
  $$
  so recalling that $(B \times B, \Boxed{\alex^{B \times B}} \cap \Boxed{\lex^{B \times B}}) = (B \times B, \Boxed{\sqsubseteq^{B \times B}})
  = (B, \Boxed{\le_1^B}) \times (B, \Boxed{\le_2^B})$, we conclude that the following holds in the category of posets of dimension
  at most~2:
  \begin{center}
    \begin{tikzcd}
	  (A, \Boxed{\preceq^A}) \arrow[d, hookrightarrow, "f"'] \arrow[r, hookrightarrow, "e"] & (A, \Boxed{\le_1^A}) \times (A, \Boxed{\le_2^A}) \arrow[r, hookrightarrow, "g"] & (B, \Boxed{\preceq^B}) \arrow[dll, hookrightarrow, dashed, "k \circ \ell"]\\
	  (\QQ, \Boxed{\preceq})
	\end{tikzcd}
  \end{center}
  This concludes the proof.
\end{proof}

\paragraph{Open problem 1.}
Extend the main results of Section~6 and this Corrigendum to classes of relational
structures of the form $(A, \Boxed{<_1}, \ldots, \Boxed{<_n})$ where
$\Boxed{<_1}$, \ldots, $\Boxed{<_n}$ are linear orders on $A$, $n \ge 3$.

\bigskip

The occurrence of the weak amalgamation property in this setting is not incidental.
Indeed, suppose that $\calK$ is a countable relational structure with big Ramsey degrees.
It then follows by compactness that the finite structures in $\Age(\calK)$ have small Ramsey degrees.
Since the joint embedding property is immediate for $\Age(\calK)$, Theorem 3.2 of \cite{masul-zucker-2024}
implies that $\Age(\calK)$ has the weak amalgamation property and
we are precisely in the framework considered here.

\paragraph{Open problem 2.} Find more examples of weak \Fraisse\ limits with big Ramsey degrees.

\section*{Addendum}

The \emph{profile} of a class $\KK$ of finite relational structures is the integer-valued function $\phi_{\KK} : \NN \to \NN$
that assigns to each nonnegative integer $n$ the number of structures in $\KK$ with $n$ elements, counted up to isomorphism.
The behavior of this function has been extensively studied, particularly in the case where $\KK$ is \emph{hereditary}
(i.e., closed under taking substructures) and consists of objects such as graphs (directed or undirected), tournaments,
ordered sets, ordered graphs, or ordered hypergraphs. Moreover, a result of Cameron~\cite{Cameron2003} shows that the study
of permutations—originating from the Stanley--Wilf conjecture and its resolution by Marcus and Tardos~\cite{MarcusTardos2004}—fits naturally into the framework of profiles of hereditary classes of ordered relational structures (see~\cite{OudrarPouzet2011,OudrarPouzet2016}).

These results demonstrate that the profile cannot exhibit arbitrary growth: there are ``jumps'' in the possible growth rates. Typically, the growth is either polynomial or faster than any polynomial (see~\cite{Pouzet1978,Pouzet2006}). For several classes of structures, the profile grows at least exponentially (e.g., tournaments~\cite{BaloghBollobasMorris2007,BoudabousPouzet2010}, ordered graphs and hypergraphs~\cite{BaloghBollobasMorris2006a,Klazar2008}, and permutations~\cite{KaiserKlazar2003}) or at least as fast as the partition function (e.g., graphs~\cite{BaloghBollobasSaksSos2009}). For further details, see the survey by Klazar~\cite{Klazar2010}.  

As noted in~\cite{oudrar-pouzet-2026}, there are very few cases for which the profile $\phi_\KK$ of a hereditary class
$\KK$ is actually a polynomial. This is why we introduce the following relaxation: we say that $\phi_\KK$ is
\emph{eventually polynomial} if for some non-negative integer $n$, the restriction of $\phi_\KK$ to
$\{n, n+1, n+2, \ldots\}$ is a polynomial function.
In particular, Pouzet and Oudrar proved the following in \cite{oudrar-pouzet-2026}, where an \emph{interval decomposition}
of an ordered relational structure is a monomorphic decomposition of the structure into intervals:

\begin{THMD} \cite[cf.~Theorem 1.1]{oudrar-pouzet-2026}
Let $\KK$ be a hereditary class of finite ordered relational structures over a finite relational language. Then:
\begin{itemize}
    \item
      either there exists an integer $k$ such that every member of $\KK$ has an interval decomposition into at most $k+1$
      blocks, in which case $\KK$ is a finite union of ages of ordered relational structures, each having an interval decomposition
      into at most $k+1$ blocks, and the profile of $\KK$ is eventually polynomial (of degree at most $k$);
    \item
      or the profile of\/ $\KK$ is at least exponential.
\end{itemize}
\end{THMD}

\begin{THME}
  Let $\calS$ be a countable relational structure with a total order relation $<$ in its language.
  If the profile of $\Age(\calS)$ is eventually polynomial, then:
  
  $(a)$ $\calS$ has a finite monomorphic decomposition, and
  
  $(b)$ $\calS$ has finite big Ramsey degrees if and only if every substructure of $\calS$ induced by a block in its
  minimal monomorphic decomposition is chained by a chain having finite big Ramsey degrees.
\end{THME}
\begin{proof}
  $(b)$ follows from $(a)$ by Theorems~4.4 and~4.8, while $(a)$ follows from Theorem~D by an easy application of K\H onig's lemma.
  Let us quickly sketch the argument for $(a)$.
  
  Let $\calS$ be a countable first-order structure with a total order relation $<$ in its language,
  and assume that the profile of $\KK = \Age(\calS)$ is eventually polynomial. Then by Theorem~D we know that
  there exists an integer $k$ such that every member of $\KK$ has an interval decomposition into at most $k+1$
  blocks. Note that if $\calB \le \calC$ are two finite substructures of $\calS$,
  then every interval decomposition of $\calC$ into at most $k+1$ blocks has a restriction that is
  an interval decomposition of $\calB$ into at most $k+1$ blocks.
  Now, take any sequence $\calS_1 \le \calS_2 \le \calS_3 \le \ldots$ of finite substructures of $\calS$
  such that $\bigcup_{i \in \NN} \calS_i = \calS$ and construct a finitely branching tree as follows:
  the root of the tree is the trivial (formal) interval decomposition $\{\0\}$ of the empty set;
  the level $i$ of the tree consists of all interval decompositions of $\calS_i$;
  each interval decomposition of $\calS_i$ is joined by an edge to its restriction to $\calS_{i-1}$.
  By K\H onig's lemma there is an infinite branch. Since each decomposition extends the one that
  precedes it, it is clear how to take the supremum along this branch to get a
  valid interval decomposition of $\calS$ into at most $k+1$ blocks.
  Finally, recall that an interval decomposition into finitely many blocks is a special finite monomorphic decomposition.
\end{proof}

\section*{Acknowledgements}

The authors are grateful to Jan Hubi\v cka for bringing to our attention the fact
that the reduct of the generic permutation we are interested in is not the generic
partial order.

\end{document}